\RequirePackage{fix-cm}
\documentclass[smallextended]{svjour3}       
\smartqed  
\usepackage{amssymb}
\usepackage{amsfonts}%
\usepackage{amsmath}%
\usepackage{graphicx}
\begin{document}

\title{Topologically singular points in the moduli space of Riemann surfaces\thanks{Authors partially supported by the project MTM2014-55812-P 
of Ministerio de Economia y Competitividad (Spain)}
}


\author{Antonio F. Costa         \and
        Ana M. Porto
}


\institute{Antonio F. Costa \at
              Departamento de Matem\'{a}ticas Fundamentales, Facultad de Ciencias, UNED,  Senda del rey, 9,  28040 Madrid, Spain \\
              Tel.: +34 91 3987224,   \email{acosta@mat.uned.es}\\      
           \and
           Ana M. Porto \at
              Departamento de Matem\'{a}ticas Fundamentales, Facultad de Ciencias, UNED,  Senda del rey, 9,  28040 Madrid, Spain \\
              Tel.: +34 91 3987233, \email{asilva@mat.uned.es}\\      
                }
\dedication{To our friend Mar\'ia Teresa Lozano}
\date{Received: date / Accepted: date}

\maketitle

\begin{abstract}
In 1962 E. H. Rauch found the points in the moduli space of Riemann surfaces not having a neighbourhood homeomorphic to a ball. These points are called here topologically singular. We give a different proof of some of the results of Rauch and also determine the topologically singular points in the branch locus of some equisymmetric families of Riemann surfaces.

\end{abstract}
\section{Introduction}

Let $M$ be a manifold and $p:M\rightarrow N$ be a regular branched covering;
$N$ has then a structure of (good) orbifold. The set of singular values of $p$
is called the branch locus and it is the image by $p$ of the fixed points of
automorphisms of the covering $p$; it consists of both ordinary manifold
points and of points we call topologically singular points, meaning that they
do not admit a neighbourhood homeomorphic to a ball. Note that all the points outside
the branch locus are manifold points.

We shall assume, all over this paper, that $g$ is an integer $\geq2$. The moduli
space $\mathcal{M}_{g}$ of surfaces of genus $g$ is endowed with the structure of an
orbifold given by the Teichm\"{u}ller space $\mathbb{T}_{g}$ and the action of
the mapping class group that produces a covering $\mathbb{T}_{g} \rightarrow \mathcal{M}_{g} $. 
In \cite{Rau} Rauch proves, that for $g>3$, every
point in the branch locus $\mathcal{B}_{g}$ of $\mathcal{M}_{g}$ is
topologically singular, the branch loci $\mathcal{B}_{2}$ and $\mathcal{B}%
_{3}$ containing topologically non-singular and singular points. In this article, we present a topological proof of these results. 

Finally, we study some equisymmetric families of dimension 4, showing how both 
topologically singular and non-singular points appear in the branch loci of moduli spaces of these families.

Acknowledgements. We wish to thank the referees for comments and suggestions.

\section{Preliminaries}

\subsection{Uniformization of Riemann surfaces and automorphisms using
Fuchsian groups}

A Fuchsian group $\Delta$ is a discrete subgroup of $\mathrm{PSL}%
(2,\mathbb{R})$, i.e. the group $\mathrm{Isom}^{+}(\mathbb{H}^{2})$ of direct
isometries of $\mathbb{H}^{2}$. If $\mathbb{H}^{2}/\Delta$ is compact, the
algebraic structure of $\Delta$ is given by the signature $s=(h;m_{1}%
,...,m_{r})$, where $h$ is the genus of the quotient surface $\mathbb{H}%
^{2}/\Delta$ and the $m_{i}\,\ $are the branched indices of the covering
$\mathbb{H}^{2}\rightarrow\mathbb{H}^{2}/\Delta$ (the order of the isotropy
groups of the conic points of the orbifold $\mathbb{H}^{2}/\Delta$). The group
$\Delta$ admits a canonical presentation:
\[
\left\langle a_{i},b_{i};i=1,...,h;x_{j};j=1,...,r:x_{1}...x_{r}%
{\textstyle\prod\nolimits^{h}}
[a_{i}b_{i}]=x_{j}^{m_{j}}=1\right\rangle
\]

We shall consider only compact Riemann surfaces. A Riemann surface $X$ of
genus $g>1$, may be uniformized by a surface Fuchsian group, i.e.
$X=\mathbb{H}^{2}/\Gamma$, where $\Gamma$ is a Fuchsian group with signature
$(g;)$, surface group of genus $g$. The group $\Gamma$ is isomorphic to the
fundamental group of $X$.

When $g>1$, the group of automorphisms of the Riemann surface $X$ is a finite
group $\mathrm{Aut}(X)$. If $G\leq\mathrm{Aut}(X)$ the quotient orbifold $X/G$
is isomorphic to $\mathbb{H}^{2}/\Delta$, where $\Delta$ is a Fuchsian group
containing $\Gamma$ and such that $\Delta/\Gamma\cong G$.

If we have an $n$-fold covering $X=\mathbb{H}^{2}/\Gamma\rightarrow
\mathbb{H}^{2}/\Delta$, where $\Gamma$ is a surface genus $g$ Fuchsian group and
$\Delta$ has signature $(h;m_{1},...,m_{r})$, the following Riemann-Hurwitz
formula holds:%
\[
2g-2=n(2h-2+%
{\displaystyle\sum\nolimits^{r}}
(1-\frac{1}{m_{j}}))
\]

\subsection{Teichm\"{u}ller and moduli spaces}

Let $\mathcal{G}$ be an abstract group isomorphic to\textbf{\ }a\textbf{\ }%
Fuchsian group with signature $s$. Two representations $\alpha_{1}$ and
$\alpha_{2}$ of $\mathcal{G}$ in $PSL(2,\mathbb{R})$ are equivalent if there
is $\gamma\in\mathrm{PSL}(2,\mathbb{R})$ such that $\alpha_{1}(\zeta
)=\gamma\alpha_{2}(\zeta)\gamma^{-1}$, for all $\zeta\in\mathcal{G}$. 
The Teichm\"{u}ller space $\mathbb{T}_{s}$ is the space of
equivalence classes of representations $\rho$ of $\mathcal{G}$ in
$\mathrm{PSL}(2,\mathbb{R})$ such that $\rho(\mathcal{G})$ is a Fuchsian group
with signature $s$. This space with the topology induced by $\mathrm{PSL}%
(2,\mathbb{R})$ is homeomorphic to a ball of dimension%
\[
\dim\mathbb{T}_{(h;m_{1},...,m_{r})}=6h-6+2r
\]

Note that the group $\mathcal{G}$ is isomorphic to the orbifold fundamental group of
$\mathbb{H}^{2}/\rho(\mathcal{G)}$, where $[\rho]\in
\mathbb{T}_{s}$. If $s=(g;)$ the corresponding Teichm\"{u}ller space is noted
$\mathbb{T}_{g}$. Let $\mathrm{Mod}_{g}$ be the mapping class group of
surfaces of genus $g$. The group $\mathrm{Mod}_{g}$ acts by composition on
$\mathbb{T}_{g}$ and the quotient $\mathbb{T}_{g}/\mathrm{Mod}_{g}%
=\mathcal{M}_{g}$ is the moduli space of Riemann surfaces of genus $g$. Note
that $\mathcal{M}_{g} $ is, by construction, an orbifold and its universal
covering is $\Pi:\mathbb{T}_{g}\rightarrow\mathcal{M}_{g}$. The set of branch
values of the covering $\Pi$ is the branch locus $\mathcal{B}_{g}$ of the
orbifold $\mathcal{M}_{g}$. The branch locus $\mathcal{B}_{g}$ is the image by
$\Pi$ of the fixed points by finite subgroups of $\mathrm{Mod}_{g}$ and
represents in $\mathcal{M}_{g}$ the surfaces with non-trivial automorphism
group (up to the exception of $\mathcal{M}_{2}$, since $\mathcal{B}_{2}$
consists of surfaces having non-trivial automorphisms different from the
hyperelliptic involution).

Let $\theta:\mathcal{G}\rightarrow G$ be an epimorphism from the abstract
group $\mathcal{G}$ isomorphic to a Fuchsian group with signature
$s=(h;m_{1},...,m_{r})$ and such that $\ker\theta$ is isomorphic to a surface
group of genus $g$. There is a natural embedding $i_{\theta}:\mathbb{T}%
_{s}\rightarrow\mathbb{T}_{g}$. The image $\Pi(i_{\theta}(\mathbb{T}%
_{s}))\subset\mathcal{M}_{g}$ consists of the surfaces of genus $g $ having a
subgroup of their automorphism groups isomorphic to $G$ with a specific action
determined by $\theta$. We say that $\Pi(i_{\theta}(\mathbb{T}_{s}))$ is the
moduli space of an equisymmetric family given by $\theta$. If we consider the
subset $S_{s,\theta}$ of $\Pi(i_{\theta}(\mathbb{T}_{s}))$ consisting of the
surfaces whose full automorphism group is $G$, we obtain a stratification of
$\mathcal{B}_{g}$ by the sets $S_{s,\theta}$ which are called the
equisymmetric strata (see \cite{BR}).

\bigskip

\section{Topologically singular points in moduli space}

\begin{definition}
(Topological singular point) A point $X$ in $\mathcal{B}_{g}$ is topologically
singular if $X$ has not a neighbourhood in $\mathcal{M}_{g}$ homeomorphic to a
$(6g-6)$ ball.
\end{definition}

In other words:

\begin{definition}
(Rauch definition \cite{Rau}) A point $X$ in $\mathcal{B}_{g}$ is singular if
$X$ is not a manifold (or uniformizable) point in $\mathcal{M}_{g}$.
\end{definition}

\begin{theorem}
\label{FundLem} The group $\mathrm{Aut}(X)$, for $X$ in the branch loci of
$\mathcal{B}_{g}$, acts as a subgroup of $\mathrm{O}(6g-6)$ in $\mathbb{S}%
^{6g-7}$. The point $X$ is topologically non-singular if and only if
$\mathbb{S}^{6g-7}/\mathrm{Aut}(X)$ is homeomorphic to $\mathbb{S}^{6g-7}$.
\end{theorem}

\begin{proof}
Let $X\in\mathcal{B}_{g}$ and $Y\in\Pi^{-1}(X)\subset\mathbb{T}_{g}$. Since
$\mathrm{Mod}_{g}$ acts discontinuously on $\mathbb{T}_{g}$, there is a
$(6g-6)$ ball $U\subset\mathbb{T}_{g}$, with center $Y$, such that
$h\in\mathrm{Mod}_{g}$ satisfies the condition $h(U)\cap U\neq\varnothing$ if
and only if $h$ fixes $Y$.

An element $h$\ of $\mathrm{Mod}_{g}$ such that $h(U)\cap U\neq\varnothing$ is
given by an automorphism of the Riemann surface represented by $Y$ in
$\mathbb{T}_{g}$, consequently by $X$ in $\mathcal{M}_{g}$, and we may 
identify $h$ with an element (still called $h$) of $\mathrm{Aut}(X)$. Under
such identification, $\mathrm{Aut}(X)$ acts on the ball $U$ and it follows,
since $h$ is an isometry of $\mathbb{T}_{g}$ (with the Teichm\"{u}ller
metric), that $h$ acts as an isometry on both $U$ and $\partial U$ ($\cong$
$\mathbb{S}^{6g-7}$); therefore, $\mathrm{Aut}(X)$ acts as a subgroup $G$ of
$\mathrm{O}(6g-6)$.

Now $\Pi(U)=U/\mathrm{Aut}(X)=U/G$ is a neighbourhood of $X$, so $X$ is
non-singular whenever $U/G$ is homeomorphic to the $(6g-6)$ ball or,
equivalently, whenever $\Pi(\partial U)\cong\mathbb{S}^{6g-7}/G$ is the sphere
$\mathbb{S}^{6g-7}$.
\end{proof}

The following theorem will be used to produce our proof of the main Theorem of \cite{Rau}.

\bigskip

\begin{theorem}
\label{TopCrit}Let $X\in\mathcal{B}_{g}$. If for each equisymmetric stratum $S$ such that $X\in\overline
{S}$, the codimension of $S$ is greater than $2$, then $X$ is topologically singular.
\end{theorem}

\begin{proof}
Let $\Pi:\mathbb{T}_{g}\rightarrow\mathcal{M}_{g}$ be the covering given by
the action of $\mathrm{Mod}_{g}$. 

Let $ \mathcal{S} $ the set of equisymmetric strata $S$ such that $X\in\overline{S}$.
Let $U$ be a ball containing a point of
$\Pi^{-1}(X)$ and such that $\Pi(U)$ does not cut
$\mathcal{B}_{g}$ but on the strata in $ \mathcal{S} $. We
shall show that if $\mathrm{codim}(S)>2$ for all $ S\in \mathcal{S}$ then $U/\mathrm{Aut}(X)$ is
not a ball.

We consider the covering $p=\Pi\mid_{\Pi^{-1}(\partial U)}:\partial
U\rightarrow\Pi(\partial U)$ with branch locus $\Pi(\partial U)\cap
\mathcal{B}_{g}=\Pi(\partial U)\cap (\cup_{\mathcal{S}} S)$. There is a finite triangulation of
$U$ in such a way that $\mathrm{Fix}(\mathrm{Aut}(X))$ is a subpolyhedron and the action of $\mathrm{Aut}(X)$ preserves the
triangulation (note that $\mathrm{Aut}(X)$ acts as a finite order rotation group of $\mathrm{O}(6g-6)$). Since $\mathrm{codim}(S)>2$, for all $ S\in \mathcal{S}$, the
codimension of the polyhedron $\Pi(\partial U)\cap\mathcal{B}_{g}=\Pi(\partial U)\cap (\cup_{\mathcal{S}} S)$ is greater than 2 in $\Pi(\partial U)$. If $\Pi(\partial U)$ is a manifold then $\pi_{1}(\Pi(\partial U)-\Pi(\partial
U)\cap\mathcal{B}_{g})\cong\pi_{1}(\Pi(\partial U))$ (see \cite{G} Theorem 2.3, page 146). If $g>2$, the covering
$p$ has $\left\vert Aut(X)\right\vert \neq1$ sheets (in case $g=2$ the covering $p$ has
$\left\vert Aut(X)/\left\langle h\right\rangle \right\vert \neq1$ sheets,
where $h$ is the hyperelliptic involution) and $p$ is branched on
$\Pi(\partial U)\cap\mathcal{B}_{g}$. In both cases $\pi_{1}(\Pi(\partial
U)-\Pi(\partial U)\cap\mathcal{B}_{g})$ must be not trivial. Hence either 
$\pi _{1}(\Pi (\partial U))\neq 1$ or $\Pi(\partial U)$
is not a manifold and, in both cases, $\Pi (\partial U)$ cannot be homeomorphic to the sphere $\mathbb{S}^{6g-7}$; 
therefore, $X$ is topologically singular.

\end{proof}

\begin{corollary}\label{Isol}
Let $X\in\mathcal{B}_{g}$ and suppose that $X$ is
isolated in $\mathcal{B}_{g}$ (see \cite{Kul}) then $X$ is topologically singular.
\end{corollary}

\begin{proof}
If $X$ is isolated then $\lbrace X \rbrace$ is the unique equisymmetric stratum $S$ such that $X\in\overline
{S}$ and $\lbrace X \rbrace$ has dimension $0$ (codimension greater than $2$). 
\end{proof}

\begin{remark}
Note that if $X$ is an isolated point of $\mathcal{B}_{g}$ and $U$ is a
ball in $\mathbb{T}_{g}$ containing a point of $\Pi^{-1}(X)$ such that $\Pi(U)\cap\mathcal{B}_{g}=\{X\}$, 
the covering $\partial U\rightarrow\Pi(\partial U)$ is a regular
unbranched covering with deck transformation group $\mathrm{Aut}(X)$, so $\Pi(\partial U))$ is a manifold and 
$\pi_{1}(\Pi(\partial U)) \cong \mathrm{Aut}(X)$.
\end{remark}

\bigskip

\begin{theorem}
\label{MainRauch}For $g\geq4$, every point in $\mathcal{B}_{g}$ is
topologically singular.
\end{theorem}

\begin{proof}
Note that every stratum of $\mathcal{B}_{g}$ is contained in the closure
$\overline{S}_{k}$ of some equisymmetric stratum defined by the action of a
prime order automorphism (for instance see \cite{C}, \cite{CI}, \cite{CIP}). We shall then study the dimension of such strata. 
Note that in some cases the full group of automorphisms of the surfaces in $\overline{S}_{k}$ 
may contain strictly the cyclic group generated by the
prime order automorphism. 

Assume that $\overline{S}_{k}$ is the closure of an equisymmetric stratum of
$\mathcal{B}_{g}$ defined by a cyclic subgroup $C_{k}$ of $\mathrm{Mod}_{g}$
of prime order $k$ such that the corresponding action on surfaces has $r$ fixed points. Let $h$ be the genus of all the quotients
of the surfaces in $S_{k}$ by the action of the group $C_{k}$. The
Riemann-Hurwitz formula gives:%
\[
2g-2=k(2h-2+(%
{\displaystyle\sum\nolimits^{r}}
(1-\frac{1}{k}))
\]
Since $k\geq2$ then
\[
2g-2\geq2(2h-2+r/2)
\]
and $r\leq2g-4h+2$. Hence the dimension of $S_{k}$ is smaller than $4g-2h-2$.
Supposing the codimension of $S_{k}$ is two, we have $6g-8=4g-2h-2$ yielding
$g=2$ or $3$, which contradicts our hypothesis. We thus conclude that the codimension of
$S_{k}$ is greater than $2$ and Theorem \ref{TopCrit}\ completes the proof.
\end{proof}

\begin{remark}
The proof in \cite{Rau} of the Theorem \ref{MainRauch} is based on a Theorem
of Zariski \cite{Z}.
\end{remark}

\bigskip

Let us now study the topologically singular points of the moduli space of
surfaces genus 2 and 3. 

We shall say that an equisymmetric stratum $S$ is maximal if $S\not\subset \overline{S'}$ for every equisymmetric stratum $S'$ with $S\neq S' $. We need the following corollary of Theorem \ref{TopCrit}:

\begin{corollary}\label{Cor} 
Let $X\in\mathcal{B}_{g}$ and suppose that $X\in\overline
{S}$ where $S$ is a maximal equisymmetric stratum of codimension greater than $2$. Then $X$ is topologically singular.
\end{corollary}

\begin{proof}
Let $X\in\overline
{S}$ where $S$ is a maximal equisymmetric stratum of codimension greater than $2$. By Theorem \ref{TopCrit}, each point $ Y \in S $ is singular since $Y$ is the unique stratum containing $Y$. Since $X\in\overline{S}$ we have $X$ is singular. 
\end{proof}

First we consider the case $g=3$.

\begin{theorem}

The points of $\mathcal{B}_{3}$ corresponding to surfaces having an
automorphism different from the hyperelliptic involution are topologically
singular, while the points of $\mathcal{B}_{3}$ corresponding to surfaces having
only the hyperelliptic involution are non-singular.
\end{theorem}

\begin{proof}

The orders of prime order automorphisms of surfaces of genus $3$ are $2,3$ and $7$ (see \cite{BCIP}). Let $S_{k}^{(i)}$ be the equisymmetric strata corresponding to genus $3$ surfaces where $C_{k}$ acts in a determined topological way, $k=2,3,7$.
Let us denote $\overline{S}_{2}^{(1)}$ the hyperelliptic
locus. Every point in $\mathcal{B}_{3}$ is in the closure of some $S_{k}^{(i)}$.

Order 2: there are three topological types of automorphisms.

Type A: the hyperelliptic involution with $8$ fixed points and quotient of genus
$0$. The stratum $S_{2}^{(1)}$ corresponding to this topological type of action has
dimension $6\times0-6+2\times8=10$ (codimension two in $\mathcal{M}_{3}$). Each point in $S_{2}^{(1)}$ has a neighbourhood that is homeomorphic
to the quotient space of a ball by the action of an order two rotation with fixed point set of codimension $2$, then the points in  $S_{2}^{(1)}$ are not
topologically singular. 

Type B: the involution has $4$ fixed points and the quotient has genus $1$. In this
case the stratum $S_{2}^{(2)}$ has dimension $6\times1-6+2\times4=8$
(codimension $>2$). Or the signature $(1; 2,2,2,2)$ is maximal (see for instance \cite{CIP}), then this stratum is maximal.
By Corollary \ref{Cor} all the points in $\overline{S}_{2}^{(2)}$ are singular.

Type C: for this type of automorphisms there are no fixed points and the quotient
surface has genus $2$. This stratum is not maximal since all Riemann surfaces of genus three with a fixed point free involution are hyperelliptic and as well admit an involution with four fixed points (the signature of Fuchsian groups determining this stratum is not maximal, see \cite{CIP}). Hence $S_{2}^{(3)}$ is contained in the closure of $S_{2}^{(2)}$
and all its points are singular.

Order 3: there are two topological types of automorphisms of order three.

Type A: two fixed points and genus of quotient $1$. This stratum $S_{3}^{(1)}$ is not maximal and it is contained in $\overline{S}_{2}^{(2)}$,
so all the points corresponding to surfaces with this type of action are singular.  

Type B: five fixed points and genus of quotient $0$. The dimension of this stratum $S_{3}^{(2)}$
is $6\times0-6+10=4$ (codimension $>2$). The stratum $S_{3}^{(2)}$ is maximal and all the points in the closure of $S_{3}^{(2)}$ are singular.  

Order 7: There are two points: $K$ and $W$. The point $K$ corresponds to the Klein quartic and the point $W$ to the Wiman's curve of type I in genus 3  (see \cite{Wi}). The surface $K$ admits order three automorphisms, then it is in $\overline{S}_{3}^{(i)}$ and $K$ is a singular point. The point $W$ is in $\overline{S}_{2}^{(1)}$ with $\mathrm{Aut}(W)=C_{14}$ and it corresponds to the curve of equation $y^{2}=x^{7}-1$. If $U$ is a neighbourhood of $W'$, where $W'$ is in the preimage of $W$ in Teichm\"{u}ller space, and $\mathrm{Aut}%
(X)=C_{14}=\left\langle t:t^{14}=1\right\rangle $, we have that $\partial
U/\left\langle t^{2}\right\rangle $ is the sphere or $\partial U/\left\langle
t^{2}\right\rangle \rightarrow\partial U/\left\langle t\right\rangle $ is an unbranched cyclic
7-fold covering; then $\partial U/\left\langle t\right\rangle $ is not simply connected and $W$
is topologically singular.

Finally, we must consider the points in $\overline{S}_{2}^{(1)}$ but not in any $S_{k}^{(i)}$. 
These points correspond to hyperelliptic surfaces having an automorphism of order $4$ that is a square root of the hyperellipticity. Let $S_{4}\subset\overline{S}_{2}^{(1)}$ be
the stratum corresponding to surfaces with full automorphism group $C_{4}$. The codimension of $S_{4}$ is greater than two 
($\dim S_{4} = 6 \times 0 - 6 + 2 \times 5 = 4$). Let $X$ be a point in $S_{4}$ and $Y\in\Pi^{-1}(X)$. Let
$U$ be a neighbourhood of $Y$ in $\mathbb{T}_{g}$ where $\mathrm{Aut}%
(X)=C_{4}=\left\langle t:t^{4}=1\right\rangle $ acts. We have that $\partial
U/\left\langle t^{2}\right\rangle $ is the sphere but $\partial U/\left\langle
t^{2}\right\rangle \rightarrow\partial U/\left\langle t\right\rangle $ is a
2-fold covering branched on a subpolyhedra of codimension $>2$; then, by a
similar argument to the one in the proof of Theorem \ref{TopCrit}, we prove
that $\partial U/\left\langle t\right\rangle $ is not simply connected and $X$
is thus topologically singular. Hence all the points in $\overline{S}_{4}$ are singular.

\end{proof}

\bigskip

Finally we consider the case $g=2$. In this case, using our approach we cannot give
a complete description of the topological singular points in $\mathcal{B}_{2}$.

\begin{theorem}
The points in the stratum $S_{2}$ of $\mathcal{B}_{2}$ corresponding to surfaces having
full automorphism group $C_{2} \times C_{2}$ are not topologically singular.
The isolated point of $\mathcal{B}_{2}$ corresponding to the Kulkarni surface $y^{2}=x^{5}-1$ is singular. 
\end{theorem}

The Kulkarni surface $y^{2}=x^{5}-1$ is also known as Wiman's curve of type I in genus 2 (see \cite{Wi}).  

\begin{proof}
First note that $\dim\mathcal{M}_{2}=6\times2-6=6$.

By Theorem \ref{FundLem}, the points of $S_{2}$ have a
neighbourhood $U$ that is homeomorphic to the quotient of a ball $B$ by a
rotation of order two having as fixed point set a linear subspace of
codimension two (intersection with $B$). Hence, $U$ is a ball and the points in
$S_{2}$ are not singular.

The stratum $S_{5}$ (surfaces with an order $5$ automorphism) has a single point and it is an
isolated point of $\mathcal{B}_{2}$ (see \cite{Kul}). By Corollary \ref{Isol} this point is singular.

\end{proof}
\bigskip
\begin{remark}
The points in $\mathcal{B}_{2}$ different from Kulkarni isolated surface and
the surfaces in $S_{2}$ are in strata completely included in
$\overline{S_{2}}$ (then these strata are non-maximal) and our methods do not provide information
on the singularity of such points.
\end{remark}

\section{Singular points of equisymmetric families}

In the case of (real) dimension two equisymmetric families, as for all two
dimensional orbifolds, the points in the branch locus are not topologically
singular. We shall show that in families of greater dimensions the points in
the singular locus may be either topologically singular or non-singular. In
that direction, let us study the singular points of some equisymetric families
of (real) dimension four.

\begin{example}
The families $\mathcal{W}_{q,w}$ consist of Riemann surfaces that are $q\times
w-$fold cyclic coverings of the sphere branched on five points, where $q,w>5$
are prime integers $q\neq w$. The type of coverings defining the families
$\mathcal{W}_{q,w}$ will be described now.

Let $\mathbb{T}_{(0;q,q,q,qw,w)}$ be the Teichm\"{u}ller space of groups
$\Delta$ with signature $(0;q,q,q,qw,w)$, and
\[
\left\langle x_{i},i=1,...,5:x_{i}^{q}=1,\text{ }i=1,2,3\text{, }x_{4}%
^{qw}=x_{5}^{w}=1\text{ and }x_{1}...x_{5}=1\right\rangle
\]
be a canonical presentation for these Fuchsian groups. The surfaces of the
family $\mathcal{W}_{q,w}$ are uniformized by the surface groups in the
kernel of the epimorphism $\theta:\Delta\rightarrow C_{qw}=\left\langle
l:l^{qw}=1\right\rangle $, given by $\theta(x_{i})=l^{w}$, $i=1,2,3$ and
$\theta(x_{4})=l^{-3w}l^{q}$, $\theta(x_{5})=l^{-q}$. The inclusion
$\ker\theta\subset\Delta$ induces an embedding $e:\mathbb{T}_{(0;q,q,q,qw,w)}%
\rightarrow\mathbb{T}_{g}$ and the moduli space of $\mathcal{W}_{q,s}$ is
$\Pi(e(\mathbb{T}_{(0;q,q,q,qw,w)}))$. The branch locus of the family consists
of one point $X$ with isotropy group $C_{3}$. This point corresponds to the
case where the group $\Delta$ is inside the triangular group $(0;q,3qw,3w)$.
The point $X$ is singular and the boundary of a neighbourhood of $X$ is
homeomorphic to the lens space $L(3,1)$ (see a survey on lens spaces in \cite{W}, there,
precisely, the lens spaces are studied as quotient singularities).
\end{example}

\begin{example}
The family $\mathcal{Q}$ consists of Riemann surfaces that are $q$-fold (where
$q>5$ is a prime) cyclic coverings of the sphere branched on five points. The
$q$-cyclic coverings defining the family have some special types which we
shall describe in terms of Fuchsian groups.

Let $\mathbb{T}_{(0;q,\overset{5}{...},q)}$ be the Teichm\"{u}ller space of
groups $\Delta$ with signature $(0;q,\overset{5}{...},q)$, and
\[
\left\langle x_{i},i=1,...,5:x_{i}^{q}=1,\text{ }i=1,...,5\text{ and }%
x_{1}...x_{5}=1\right\rangle
\]
be a canonical presentation for these Fuchsian groups. The family
$\mathcal{Q}$ has dimension four, the surfaces of the family are uniformized
by the surface groups in the kernel of the epimorphism $\theta
:\Delta\rightarrow C_{q}=\left\langle l:l^{q}=1\right\rangle $, given by
$\theta(x_{i})=l$, $i=1,...,3$ and $\theta(x_{4})=\theta(x_{5})=l^{\frac
{q-3}{2}}$. The inclusion $\ker\theta\subset\Delta$ induces an embedding
$e:\mathbb{T}_{(0;q,q,q,q,q)}\rightarrow\mathbb{T}_{g}$ and the moduli space
of $\mathcal{Q}$ is $\Pi(e(\mathbb{T}_{(0;q,\overset{5}{...},q)}))$. The
branch locus of the family consists of a dimension two subset $\mathcal{L}$ corresponding to
cone points with isotropy group of order $2$ and one point $Y$ with isotropy
group $D_{3}$. The points in $\mathcal{L}$ correspond to Fuchsian groups
$\Delta$ \ in $\mathbb{T}_{(0;q,\overset{5}{...},q)}$ contained in Fuchsian
groups $\Lambda$ with signature $(0;2,q,q,2q)$ and the point $Y$ corresponds
to the case where the group $\Delta$ is inside the triangular group
$(0;2,2q,3q)$. The points in $\mathcal{L}$ have a neighbourhood $U$ such that
the boundary is the tridimensional sphere, since the covering
\[
\Pi(e(\mathbb{T}_{(0;q,\overset{5}{...},q)})\cap\Pi^{-1}(\partial
U))\rightarrow\partial U
\]
is given by the quotient of a rotation around a trivial knot. The singular
point $Y$ has also a neighbourhood $V$ whose boundary is homeomorphic to
$\mathbb{S}^{3}$, because the intersection $\partial V\cap\mathcal{L}$ is the
trefoil knot and the covering
\[
\Pi(e(\mathbb{T}_{(0;q,\overset{5}{...},q)})\cap\Pi^{-1}(\partial
U))\rightarrow\partial U
\]
is the universal covering of the orbifold defined on $\mathbb{S}^{3}$ with
singular orbifold locus the trefoil knot with isotropy group $C_{2}$ for the
points in the branch locus. The covering is equivalent to the composition of
coverings $\mathbb{S}^{3}\overset{3:1}{\rightarrow}%
L(3,1)\overset{2:1}{\rightarrow}\mathbb{S}^{3}$ (the first covering is unbranched and the second one is the given by the Montesinos involution). 
Note that in $L(3,1)$ there are two types of involutions (see \cite{HR}): one of them being represented by the Montesinos involution and the other one given by circle action but in this case, the lifts of this involution together with the order three automorphism produce no a dihedral action on $\mathbb{S}^{3}$.

\end{example}

We feel very happy to conclude this article in honour of Professor Maite Lozano with
this application of the theory of branched coverings of $3-$manifolds.


\begin{thebibliography}{9}                                                                                                %

\bibitem {BCIP}Bartolini, G.; Costa, A. F.; Izquierdo, M.; Porto, A. M. On the
connectedness of the branch locus of the moduli space of Riemann surfaces.
Rev. R. Acad. Cienc. Exactas F\'{\i}s. Nat. Ser. A Math. RACSAM 104 (2010),
no. 1, 81--86.

\bibitem {BR}Broughton, S. A. The equisymmetric stratification of the moduli
space and the Krull dimension of mapping class groups. Topology Appl. 37
(1990), no. 2, 101--113.

\bibitem {C} Cornalba, M. On the locus of curves with automorphisms. Ann. Mat. Pura Appl. (4) 149 (1987), 135--151.

\bibitem {CI} Costa, A. F.; Izquierdo, M. On the connectedness of the branch locus of the moduli space of Riemann surfaces of genus 4. Glasg. Math. J. 52 (2010), no. 2, 401–-408. 

\bibitem {CIP} Costa, A. F.; Izquierdo, M.; Porto, A. M. Maximal and Non-maximal NEC and Fuchsian groups
uniformizing Klein and Riemann surfaces. In Riemann and Klein Surfaces, Automorphisms, Symmetries and Moduli Spaces,
Contemporary Mathematics 629, American Mathematical Society, Providence (RI) USA, 2014. DOI: http://dx.doi.org/10.1090/conm/629.

\bibitem {G}Godbillon, C. \'{E}lements de topologie alg\'{e}brique, Hermann, Paris 1971.

\bibitem {HR}Hodgson, C.; Rubinstein J.H. Involutions and isotopies of Lens Spaces. In Knot Theory and Manifolds: Proceedings of a Conference held in Vancouver, Canada, June 2--4, 1983, Springer Berlin Heidelberg, 1985, 60--96.

\bibitem {Kul}Kulkarni, R. S. Isolated points in the branch locus of the
moduli space of compact Riemann surfaces. Ann. Acad. Sci. Fenn. Ser. A I Math.
16 (1991), no. 1, 71--81.

\bibitem {Rau}Rauch, H. E. The singularities of the modulus space. Bull. Amer.
Math. Soc. 68 (1962) 390--394.

\bibitem {W}Weber, C. Lens spaces among 3-manifolds and quotient surface
singularities, Preprint 2017, to appear in RACSAM in this volume.

\bibitem {Wi} Wiman, A. Über die hyperelliptischen Curven und diejenigen von Geschlechte p=3 Jwelche eindeutige Transformationen in sich zulassen. - Bihang till K. Svenska Vet.-Akad. Handlingar, Stockholm 1895--6, bd. 21, 1--28.

\bibitem {Z}Zariski, O. On the purity of the branch locus of algebraic
functions. Proc. Nat. Acad. Sci. U.S.A. 44 (1958) 791--796.
\end{thebibliography}
\end{document}